\documentclass{article}
\usepackage{enumitem}
\usepackage{libertine}
\usepackage[libertine]{newtxmath}
\usepackage{caption}
\usepackage{subcaption}

\allowdisplaybreaks
\usepackage{xifthen}
\usepackage{physics}
\usepackage{hyperref}
\usepackage{cleveref}
\usepackage{amsthm,thmtools}
\usepackage{xcolor}
\usepackage{comment}
\usepackage{graphicx}
\usepackage{tcolorbox}
\usepackage{algorithm}
\usepackage{bbm}
\usepackage{algorithmic}
\theoremstyle{definition}
\newtheorem{theorem}{Theorem}[section]

\newtheorem{lemma}[theorem]{Lemma}

\newtheorem{definition}{Definition}[section]

\newtheorem{corollary}[theorem]{Corollary}

\newtheorem{claim}[theorem]{Claim} 

\let\originalleft\left
\let\originalright\right
\renewcommand{\left}{\mathopen{}\mathclose\bgroup\originalleft}
\renewcommand{\right}{\aftergroup\egroup\originalright}

\newcommand{\bP}[2][]{\Pr\ifthenelse{\isempty{#1}}{}{_{#1}}\left[#2\right]}
\newcommand{\bE}[2][]{\mathop\mathbb{E}\ifthenelse{\isempty{#1}}{}{_{#1}}\left[#2\right]}
\newcommand{\bI}[2][]{\mathop\mathbb{I}\ifthenelse{\isempty{#1}}{}{_{#1}}\left[#2\right]}
\newcommand{\Var}[2][]{\mathbf{Var}\ifthenelse{\isempty{#1}}{}{_{#1}}\left[#2\right]}

\allowdisplaybreaks

\usepackage{thm-restate}

\title{\textrm{Rankwidth of Graphs with Balanced Separations: Expansion for Dense Graphs}\\}
\author{{Emile Anand}\thanks{Department of Algorithms, Combinatorics, and Optimization. School of Computer Science. Georgia Institute of Technology. Atlanta, GA, 30308. Email: \textrm{emile@gatech.edu}. Work done while an undergraduate at Caltech, and while a visiting researcher at Princeton University's Department of Mathematics.}\\}

\begin{document}

\maketitle

\begin{abstract}
\noindent We prove that every graph of rankwidth at least $72r$ contains an induced subgraph whose \emph{minimum balanced cutrank} is at least $r$, which implies a vertex subset where every balanced separation has $\mathbb{F}_2$-cutrank at least $r$. This implies a novel relation between rankwidth and a well-linkedness measure, defined entirely by balanced vertex cuts. As a byproduct, our result supports the notion of \emph{rank-expansion} as a suitable candidate for measuring expansion in dense graphs.  
\end{abstract}

\section{Introduction}
  The relationship between combinatorial width parameters and connectivity properties is a cornerstone of structural and algorithmic graph theory \cite{BODLAENDER19981, SEYMOUR199322, ROBERTSON1994323, harvey2016parameterstiedtreewidth, BOTTCHER20101217}.  
For tree-width, the celebrated Grid-Minor Theorem of Chekuri and Chuzhoy \cite{10.1145/2820609} and its recent developments \cite{doi:10.1137/1.9781611977912.48, Geelen_2023} characterizes large width via the presence of an $r\times r$ grid minor, while branch-width \cite{fomin2021fastfptapproximationbranchwidth,https://doi.org/10.1002/net.20050} admits a well-linked-set characterization due to Oum and Seymour \cite{Oum_Seymour_2006}. While such relationships are well-understood for treewidth and branchwidth, a corresponding characterization for rankwidth and its sibling, clique-width, remains a gap in the literature: indeed, no well-linkedness or expander-style certificate has been pinned down. Here, we provide a proof of existence of such a certificate by establishing a linear dependence between rankwidth and another parameter, the minimum balanced cutrank. We begin by defining the minimum balanced cutrank. \\

A \emph{separation} of a graph $G=(V, E)$ is a pair $X,Y\subseteq V(G)$ such that $X\sqcup Y = V(G)$. Given a separation $(X,Y)$ of $V(G)$, the \emph{cutrank} of the separation, $\operatorname{cutrank}(X,Y)$, is the $\mathbb{F}_2$-rank of its $|X|\times|Y|$ biadjacency matrix. We say that a separation $(X,Y)$ is \emph{balanced} if $\tfrac13|V(G)|\le|X|$ and $|Y|\le\tfrac23|V(G)|$. Next, for any graph $G=(V,E)$, define the \emph{minimum balanced cutrank}, $\operatorname{min-bal-cutrank}(G)$, by
\begin{align}\operatorname{min-bal-cutrank}(G) = \min\left\{\operatorname{cutrank}(X,Y)\text{ }\Big| \text{ }X\sqcup Y = V(G)\text{ with }|X|, |Y| \leq \frac{2n}{3}\right\}. \end{align}

Informally, a graph whose minimum balanced cutrank is large should have no bisection with low rank, and therefore behaves like an \emph{expander graph} when measured in $\mathbb{F}_2$ rank. A vertex set is said to be an expander if it has many external neighbors. We call a graph bounded if every sufficiently large subgraph contains a subset which is not expanding.  There is a long line of literature on expander graphs for their ubiquitous use in pseudorandomness \cite{Alon1986Eigenvalues,anand2025pseudorandomnessstickyrandomwalk,golowich_et_al:LIPIcs.CCC.2022.27,anand2025pseudorandomnessexpanderrandomwalks}, unique games \cite{bafna2021highdimensionalexpanderseigenstripping}, log-space computation \cite{Reingold2008}, the PCP theorem \cite{10.1145/1236457.1236459}, fault-tolerant computation \cite{AjtaiAlonBruckEtAl1992,chaudhari2025peer}, and a number of other applications. We refer the interested reader to the (incredible) surveys of Lubotzky \cite{lubotzky2011expandergraphspureapplied} and Hoory, Linial, and Wigderson \cite{Hoory2006}.  However, the combinatorial notion of expansion is only meaningful for sparse graphs, and it is desirable in a number of applications to have an analog of expansion for dense graphs \cite{alon2003densegraphsantimagic, alon1992almost, doi:10.1137/18M122039X, 10.1016/0012-365X(95)00242-O, Alon_Nachmias_Shalev_2022, alon1998packing, srivastava2017alonboppanatypeboundweighted}. There are a number of concepts that have been studied in the literature previously, such as flip-width \cite{K_hn_2014,toruńczyk2024flipwidthcopsrobberdense} and twin-width \cite{bonnet2022}, but prior to now there was no agreed-upon notion.\\ 
 
For instance, dense-graph expansion could be defined by the existence of a graphon sequence of finite cut-norm with a growing number of expander subgraphs of degree $O(1)$; however, this notion is maximized by a clique, for which we instead desire a small value. Information-theoretically, it is helpful for a graph and its complement to have similar dense-graph measures of expansion.  While twin-width is defined via contraction sequences and flip-width via a game-theoretic process, \emph{rank-expansion} is based on the $\mathbb{F}_2$-rank of the adjacency matrix. 
This offers a direct linear-algebraic perspective that is intuitively similar to the spectral definitions of expansion for sparse graphs. Moreover, a good parameter for expansion in dense graphs should be small on structurally-simple dense graphs (like cliques) and large on complex, random-looking dense graphs. This is satisfied by rankwidth, since the rankwidth of a clique is $1$ while the rankwidth of a random graph $G\sim G(n,p)$ is linear in $n$ for constant $p$, with high probability. With this perspective, our main result in Theorem \ref{thm:main} proposes that \emph{rank-expansion is a suitable candidate for the analog of expansion for dense graphs}. Moreover, when working over $\mathbb{F}_2$, taking the complement of the graph is equivalent to adding the all-ones matrix, which only changes the cutrank by at most $1$.

\begin{definition}[Rank-decomposition and rankwidth \cite{Oum_Seymour_2006, Voigt}] Given a graph $G=(V, E)$, a \emph{rank-decomposition} of $G$ is a pair $(T,\gamma)$ where $T$ is a subcubic tree (i.e., every internal node has degree at most $3$) and $\gamma\colon V(G)\to\operatorname{Leaf}(T)$ is a bijection. Observe that
deleting any edge $e\in E(T)$ splits the set of leaves into parts $(X_e,Y_e)$, thereby inducing a cut in $G$.  
Then, the rankwidth of $G$ is given by
\begin{align}\mathrm{rankwidth}(G) \coloneqq \min_{(T,\gamma)} \max_{e\in E(T)}\operatorname{cutrank}_G(X_e,Y_e).\end{align}
\end{definition}

Motivated by this definition of rankwidth, we now provide a formal definition of \emph{rank-expansion}.

\begin{definition}[Rank-Expansion]\label{def: rank-expansion} 
For a graph $G=(V,E)$ where $|V|=n$, the \emph{rank-expansion} is given by
    \begin{align}\operatorname{rank-expansion}(G) \coloneqq \max_{S\subseteq V(G): |S|\leq n/2} \frac{\operatorname{cutrank}(S, V(G)\setminus S)}{|S|}.\end{align}
\end{definition}

We show that large rankwidths of a graph $G$ enforce the existence of a large, rank-robust subgraph $H$.

\begin{restatable}{theorem}{MainThm}
\label{thm:main}
Let $r\ge1$ and $G=(V, E)$ be a graph. If $\mathrm{rankwidth}(G)\ge 72r$, then $G$ contains an induced subgraph $H$ with $\operatorname{min-bal-cutrank(H)}\ge r$.
\end{restatable}

\noindent Theorem~\ref{thm:main} parallels the grid-minor theorem of   Chekuri and Chuzhoy \cite{doi:10.1137/1.9781611977912.48, Geelen_2023, 10.1145/2820609}: delete vertices instead of contracting edges, and replace edge cuts by $\mathbb{F}_2$-rank cuts. For any graph $G$, this implies that
\begin{align}
\tfrac1{72}\,\mathrm{rankwidth}(G)\;\le\;
\max_{H\subseteq G}\operatorname{min-bal-cutrank(H)}
\;\le\;
\mathrm{rankwidth}(G).
\end{align}

\noindent Here, the right inequality holds because for any graph $H$ and any rank-decomposition of $H$ of width $k$, there exists a balanced cut $(A, B)$ with $|A|,|B|\geq |V(H)|/3$ such that $\operatorname{cutrank}_H(A,B) \leq k$. Therefore, for every induced $H\subseteq G$, the minimum cutrank (across balanced sets $A$ and $B$) is $\operatorname{cutrank}_H(A,B)\le \mathrm{rankwidth}(H)\le \mathrm{rankwidth}(G)$: by properties of tree decompositions, any rank-decomposition of width $k$ contains a balanced cut of rank at most $k$. The left inequality follows from Theorem \ref{thm:main}, since if $\mathrm{rankwidth}(G)\ge 72r$, then
there is an induced $H$ satisfying $\min_{(A,B)}\operatorname{cutrank}_H(A,B)\ge r$, for some balanced partition $(A,B)$ of $V(H)$.\\

Oum’s vertex-minor theorem states that rankwidth~$k$ forces a $K_t$ minor with $t\approx2^k$ and provides a set of excluded vertex-minors for graphs of bounded rankwidth, which is analogous to the Robertson-Seymour theorem for graph minors \cite{Oum_Seymour_2006}.  
In contrast, our result finds an {induced} subgraph such that  {every} balanced separation of the subgraph has rank $\Theta(k)$. Thus, graphs with large width must contain a highly connected induced subgraph with robust cutrank. Dvořák and Norin show that large tree-width forces large minimum balanced \emph{edge}-separators~\cite{Dvorak_Norin_2018}; in our work, we show how to transfer this to vertex-rank cuts. Finally, Oum and Seymour provide approximation algorithms for rankwidth \cite{Oum_Seymour_2006}; in this spirit, Theorem \ref{thm:main} supplies a proof of existence of an alternative certificate that may lead to new parameter-testing routines. \\

Through our result, we suggest rank-expansion as a candidate notion of expansion suitable for dense graphs. Indeed, our result can analogously be viewed as a Cheeger-type inequality for dense graphs \cite{Cheeger1970,AlonMilman1985}, which adds to the growing literature in the field \cite{khetan2018cheegerinequalitiesgraphlimits,jost2024cheegerinequalitiessimplicialcomplexes,10.1145/3564246.3585139,briani2022classcheegerinequalities, 10.1145/2582112.2582118, schild2018schurcomplementcheegerinequality, first2025cheegerinequalitycoboundaryexpansion}. For instance, Cheeger gives that $c_1(1 - \lambda_2) \leq \phi(G) \leq c_2\sqrt{1-\lambda_2}$, where $\phi(G)$ is the conductance of $G$, and $\lambda_2$ is the second largest (in absolute value) eigenvalue of $G$. Conversely, Theorem \ref{thm:main} gives $\frac{1}{72}\operatorname{rankwidth}(G) \leq \rho(G) \leq \operatorname{rankwidth}(G)$, where $\rho(G)$ is the minimum balanced cutrank of any subgraph of $G$, i.e. $\rho(G) = \max_{H\subseteq G} \operatorname{min-bal-cutrank}_H(A,B)$. In this vein, our work is most related to that of \cite{K_hn_2014} which shows a different result that dense regular graphs can be partitioned into robust expanders, which are subgraphs that still have good expansion properties after removing some vertices and edges. In turn, this can be applied to Hamiltonicity problems in dense regular graphs.

\paragraph{Proof overview.} Our proof proceeds by way of contradiction. If a graph has high rankwidth but low min-bal-cutrank, then every part of the graph can be cleanly separated with a low-rank cut. We introduce the tangle in definition \ref{def: rank-tangle of order k}, which points to the `more connected' component of a separation of a graph. We use it to find a separation $(X,Y)$ such that the tangle-selected component is maximally connected. Subject to this, we assume $X$ is maximal and call it the {small side}, following the convention of \cite{Dvorak_Norin_2018}. We then \emph{compress} $X$ by selecting a subset of {representative} vertices, and show that finding a low-rank balanced cut in this compressed graph would allow us to extend $X$, which contradicts its maximality. Our proof continues via three stages.\\

First, in Claim \ref{claim:rank-tangle existence}, we adapt the tangle-and-congestion machinery of Dvořák and Norin \cite{Dvorak_Norin_2018} to extract the more strongly-connected side $X$ of a low-rank separation $(X,Y)$ that is maximally large yet \emph{well-behaved}. This well-behaved property formalizes the idea that $X$ is robustly connected to $Y$, meaning no small component of $X$ can be separated from the rest of the graph with a low-rank cut. Next, we observe that the number of distinct neighborhood patterns that vertices in $Y$ can form within $X$ is bounded. This allows us to find a small subset of representative vertices in $X$ that captures all of these connections, effectively compressing $X$ without losing essential rank information. Specifically, we compress $X$ by taking one representative per $Y$-neighborhood class (Claim \ref{claim:findingH}), which preserves $\operatorname{cutrank}(X,Y)$; then, in Lemma \ref{lemma: splitting Y reduction}, we form an induced subgraph $H'$  so that any balanced cut of $H'$ of rank at most $r$ splits $Y$ and yields two lifted separations of $G$ whose cutranks are at most $72r$; combined with the tangle axiom this yields the existence of a strongly-connected side in $V(G)$ that is strictly larger than $X$, thereby contradicting the maximality assumption.\\

\section{Preliminaries}\label{sec:prelim}

\paragraph{Graphs and matrices.} The graphs we consider are all finite, simple, and undirected.
For a graph $G=(V,E)$ and disjoint $X,Y\subseteq V$, the {biadjacency matrix} $\operatorname{Adj}_G(X,Y)$ is the $|X|\times|Y|$ Boolean matrix whose $(i,j)$'th entry is $1$ iff $(i,j)\in E$.  Unless noted otherwise, the matrix rank is taken over $\mathbb{F}_2$.

\begin{definition}[Cutrank]
For a partition $(X,Y)$ of $V(G)$, the \emph{cutrank} is
\begin{align}
\operatorname{cutrank}_G(X,Y)=\operatorname{rank}\bigl(\operatorname{Adj}_G(X,Y)\bigr).
\end{align}
\end{definition}

\begin{definition}[Balanced partition]
A partition $(X,Y)$ of $V(G)$ is \emph{$(\alpha,\beta)$-balanced} if $|X| \ge \alpha|V(G)|$ and $|Y|\le\beta|V(G)|$.  
We focus on the canonical $(1/3, 2/3)$-balanced case and refer to such partitions as \emph{balanced}. 
In other words, a partition $(X,Y)$ of graph $V(G)$ is balanced if $X\cup Y = V(G), X\cap Y = \emptyset$, and $|V(G)|/3 \leq |X|,|Y| \leq 2|V(G)|/3$.  
\end{definition}

\begin{definition}[Minimum Balanced Cutrank of Graph $G$]
\begin{align}
\operatorname{min-bal-cutrank(G)} = \min_{(X,Y) \text{ of $V(G):|X|, |Y| \geq \frac{V(G)}{3}$}}\operatorname{cutrank}(X,Y)\end{align}
\end{definition}

\noindent Now, we define some terminology that is crucial to our proof methodology.

\begin{definition}[Rank-tangle of order $k$] \label{def: rank-tangle of order k}
Following Robertson-Seymour \cite{ROBERTSON1991153}, a rank-tangle of order $k$ is a set $\mathcal{T}$ of subsets of $V(G)$ such that \begin{enumerate}\item For every partition $(X,Y)$ of $V(G)$ with $\operatorname{cutrank}_G(X,Y)<k$, exactly one of $X$ or $Y$ is in $\mathcal{T}$ (and is called the ``small side'')\item No three chosen small sides have union $V(G)$; i.e., for any sets $S_1, S_2, S_3 \in \mathcal{T}$, their union $S_1 \cup S_2 \cup S_3 \neq V(G)$.\end{enumerate}
The existence of a rank-tangle of order $k$ witnesses that $\mathrm{rankwidth}(G)\geq k$.
\end{definition}

 Given that we seek to find a separation $(X,Y)$ where $X$ is not only large but also cannot be split from $Y$ by a low-rank cut that isolates a small part of $X$, we formalize a notion of \emph{robustness} of separations.

\begin{definition}[Well-behaved separations] We say that a partition $(X,Y)$ of $V(G)$ is $(s,\epsilon)$-well-behaved if for every partition $(A,B)$ of $X$ where $\operatorname{cutrank}(A, Y) \leq 70r$ and $|A| \leq 6s$, we have that $\operatorname{cutrank}(A \cup Y,B) \geq 6\epsilon (70r)$. In general, we consider the setting of $(3|Y|, 1/7)$-well-behaved separations. \label{def: well-behaved separations}
\end{definition}

Our theorem, stated informally, is that there exists a function $f:\mathbb{N}\to\mathbb{N}$ such that $\mathrm{rankwidth}(G)$ is at most $f(\max_{H\subseteq G}  \operatorname{min-bal-cutrank}(H))$.
\noindent We dedicate Section \ref{sec: proof_approach} to the proof of the theorem. More formally, we show:
\MainThm*

To study a corollary of our result, we introduce the notions of tree-width and minimum balanced separations.

\begin{definition}[Tree Decomposition of Graph]
    A tree decomposition of a graph $G=(V,E)$ is a tree $T$ with nodes $X_1 \cup X_2 \cup \dots \cup X_n = V(G)$ satisfying:
    \begin{itemize}
        \item If $v\in X_i \cap X_j$, then all nodes $X_k$ of $T$ in the unique path between $X_i$ and $X_j$ contain $v$ as well,
        \item For all $(u,v)\in E(G)$, there is a subset $X_i$ containing $u$ and $v$.
    \end{itemize}
\end{definition}

\begin{definition}[Tree-width]
    The \emph{width} of a tree decomposition is the size of its largest set $X_i$ minus one. The tree-width\footnote{This is equivalently referred to as the order of a certain bramble in \cite{Dvorak_Norin_2018}: a set $\mathcal{B}$ of non-empty subsets of $V(G)$ is a bramble if the induced subgraph $G[S]$ is connected for all $S\in \mathcal{B}$ and the induced subgraph of $G[S_1\cup S_2]$ is connected for all $S_1,S_2\in \mathcal{B}$. Then, a graph has tree-width at least $k-1$ if and only if it has a bramble of order $k$ \cite{Dvorak_Norin_2018}.} $tw(G)$ is the minimum width among all possible tree decompositions of $G$.
\end{definition}

\begin{lemma}[Dvořák-Norin \cite{Dvorak_Norin_2018}] If every subgraph of a graph $G$ has a balanced
separation of order at most $a$, then $G$ has treewidth at most $15a$.
\end{lemma}

\begin{lemma}[Dvořák-Norin \cite{Dvorak_Norin_2018}]\label{lemma: dvorak-norin rankwidth treewidth} If rank-width(G) $\leq k$ and $G$ has no $K_{t,t}$ subgraph, then there exists a function $f:\mathbb{N}\times\mathbb{N}\to\mathbb{Z}_+$ for which $tw(G) \leq f(k, t)$.
\end{lemma}

Then, due to Lemma \ref{lemma: dvorak-norin rankwidth treewidth} and Theorem \ref{thm:main}, we arrive at the following corollary.

\begin{corollary} There exists a function $g$ so that if for every subgraph of a graph $G$ has a balanced separation of order at most $a$, then the minimum balanced cutrank of $G$ is at most $g(a)$. \\
\end{corollary}

\section{Remarks on Neighborhood Equivalence Classes}
\label{sec: equivalence class neighborhood}  

\begin{definition}
    Suppose $(X, Y)$ is a partition of $V(G)$. For each $u,v\in Y$, we define equivalence classes, where each class has the relation $u\sim v$ iff $N_G(u)\cap X = N_G(v) \cap X$. The family $\mathcal F=\{\,N_G(y)\cap X : y\in Y\}$ of subsets of $X$ has \emph{VC-dimension} $\operatorname{VC}(\mathcal F)$, the largest $S\subseteq X$ shattered by $\mathcal F$.  
Then, Sauer–Shelah's lemma \cite{SAUER1972145,pjm1102968432} bounds $|\mathcal F|$ (the number of $X$-equivalence classes) by $\sum_{i=0}^{d}\binom{|X|}{i}=O(|X|^{d})$ where $d=\operatorname{VC}(\mathcal F)$.
\end{definition}
Let the number of equivalence classes be denoted by $N_Y$. Then, we have that $\log(N_Y) \leq \operatorname{cutrank}(X,Y) \leq {N_Y}$ since the rank cannot exceed the number of distinct column types (equivalence classes), and over $\mathbb{F}_2$, $k$ dimensions can encode at most $2^k$ distinct neighborhood patterns. Furthermore, if $N_Y \leq O(n^{1-\epsilon})$, then a number of separator theorems emerge, due to the works of Feige, Hajiaghayi, and Lee  \cite{Feige_Hajiaghayi_Lee_2008}.

\noindent Let $X'$ be the minimal cardinality subset of $X$ such that \[\rank(\mathrm{Adj}(X',Y)) = \rank(\mathrm{Adj}(X,Y)).\] Equivalently, $|X'| = \mathrm{cutrank}(X,Y)$. Then, the number of equivalence classes to $X$ is equal to the number of equivalence classes to $X'$.
\begin{lemma}\label{lemma: equivalence class vc}
    The number of equivalence classes to $X$ is equal to the number of equivalence classes to $X'$.
\end{lemma}

\begin{proof} 
Let $M = \mathrm{Adj}(X,Y)$ and $M' = \mathrm{Adj}(X',Y)$. Since $|X'| = \mathrm{cutrank}(X,Y)$, the rows of $M'$ form a basis for the row space of $M$. For any two vertices $y_1, y_2 \in Y$, let $c_1$ and $c_2$ be their corresponding columns in $M$. The vertices are equivalent in $X$ if and only if $c_1 = c_2$ as vectors in $\mathbb{F}_2^X$. Since the rows of $X'$ span the row space of $M$, we have that $c_1 = c_2$ in $\mathbb{F}_2^X$ if and only if their restrictions to $X'$ are equal, i.e., if and only if the corresponding columns in $M'$ are equal. Therefore, $y_1$ and $y_2$ are equivalent with respect to $X$ if and only if they are equivalent with respect to $X'$. This establishes a bijection between the equivalence classes, so the number of equivalence classes to $X$ equals the number to $X'$.\qedhere\\
\end{proof}

\section{Proof approach}
\label{sec: proof_approach}

\MainThm*

\begin{proof}

Let $G$ be a graph with rankwidth$(G)\geq 72r$, and suppose towards a contradiction that every induced subgraph of $G$ has min-bal-cutrank strictly less than $r$. Then, there exists a rank-tangle of order $72r$. Fix one and call it $\tau$.\\

\subsection{The Maximal Well-Behaved Separation}
We begin by orienting low‑rank separations using the rank‑tangle $\tau$. Among all separations of cutrank strictly less than $70r$ that $\tau$ orients towards its ``small side,'' we choose one where the small side is maximal by size (and subject to this, has minimal cutrank). The next claim records the robustness properties that this choice enforces.\\

\begin{claim}
\label{claim:rank-tangle existence}
    There exists a partition $(X,Y)$ of $V(G)$ such that $\operatorname{cutrank}(X,Y) < 70r$, the rank-tangle $\tau$ of order $72r$ orients this partition towards $X$ (i.e., $X \in \tau$), and $(X,Y)$ is $(3|Y|, 1/7)$-well-behaved. Moreover, $|X|$ is maximum among all separations oriented by $\tau$ with $\mathrm{cutrank}<70r$ (and, subject to this, $\mathrm{cutrank}(X,Y)$ is minimum).
\end{claim}

\begin{proof} 
Since $\tau$ is a rank-tangle of order $72r$, there exists a partition $(X,Y)$ of $V(G)$ of $\operatorname{cutrank}$ at most $70r$ such that $\tau$ considers $X$ to be the small side. Of all possible partitions $(X,Y)$, choose $(X,Y)$ such that $|X|$ is maximum. Subject to this, choose $(X, Y)$ such that $\operatorname{cutrank}$$(X, Y)$ is minimum. By way of contradiction, suppose $X$ is not $(3|Y|, 1/7)$-well-behaved. \\

If $(X,Y)$ is maximal but not well-behaved, then there exists a partition $(A,B)$ of $X$ such that that $\operatorname{cutrank}(A,Y)\leq 70r$, $|A| \leq 6s$, and $\operatorname{cutrank}(A \cup Y,B) < 60r$. The tangle $\tau$ must choose one of $A\cup Y$ or $B$ as the small side.  It cannot choose $A\cup Y$ because then $X$ and $A\cup Y$ would be two small sides whose union is $V(G)$ (since $X \cup (A\cup Y) = X\cup Y = V(G)$), which violates the tangle axiom (which states that no three small sides of the tangle $\tau$ can have union $V(G)$). Therefore, $\tau$ must choose $B$ as the small side.\\

Note that $|A\cup Y|\leq |Y|+6s \leq 19|Y|$. Suppose $(X_0, Y_0) = (B, A\cup Y)$. Since $G[Y_0]$ is an induced subgraph, by our initial contradiction assumption, it has a balanced partition $(U_1^{(0)},U_2^{(0)})$ with $\operatorname{cutrank}_{G[Y_0]}(U_1^{(0)}, U_2^{(0)}) < r$. Then 
\begin{align*}\operatorname{cutrank}(X_0\cup U_1^{(0)}, U_2^{(0)}) &\leq \operatorname{cutrank}(U_1^{(0)}, U_2^{(0)})+ \operatorname{cutrank}(U_1^{(0)}, X_0) \\ &< r+ \operatorname{cutrank}(A \setminus U_2^{(0)}, B) \\ &< 61r.\end{align*}

By a symmetric argument, $\operatorname{cutrank}(U_2^{(0)}, V(G)\setminus U_2^{(0)})<61r$. Thus both of the partitions $(X_0 \cup U_1^{(0)}, U_2^{(0)})$ and $(X_0 \cup U_2^{(0)}, U_1^{(0)})$ have $\operatorname{cutrank}$ less than $61r$. For each of these partitions, the tangle $\tau$ selects one of the sides to be small. Note that either $X_0 \cup U_1^{(0)}$ or $X_0 \cup U_2^{(0)}$ is the small side of $\tau$. Otherwise, $U_2^{(0)}$ and $U_1^{(0)}$ would be small sides of $\tau$, and $V(G)$ would be the union of three small sides, $U_1^{(0)}$, $U_2^{(0)}$, and $X_0$ which results in a contradiction. So up to symmetry between $U_1^{(0)}$ and $U_2^{(0)}$, we may assume that $X_0 \cup U_1^{(0)}$ is a small side of $\tau$. \\

Then, analogously to \cite{Dvorak_Norin_2018}, given $(X_i, Y_i)$ and a balanced low-rank cut $(U_1^{(i)}, U_2^{(i)})$ of $G[Y_i]$, we can construct a sequence $(X_{i+1},Y_{i+1}) = (X_i \cup U_1^{(i)}, U_2^{(i)})$ for $i\in \{0,\dots,7\}$ where $(U_1^{(i)},U_2^{(i)})$ is a balanced partition of rank at most $r$ in $G[Y_{i}]$. This mechanism allows us to amplify the size of the sets $X_i$ to get the final contradiction. . Iteratively, $X_7\cup U_1^{(0)}$ is a small side as well. Then, $X_0\subseteq\dots\subseteq X_8$ and $Y_0 \supseteq\dots\supseteq Y_8$. Note that for $i\in \{1,\dots,7\}$, we have due to subadditivity and
\begin{align*}
    \operatorname{cutrank}(X_{i+1}, Y_{i+1}) &= \operatorname{cutrank}(X_i \cup U_1^{(i)}, U_2^{(i)} \\
    &\leq \operatorname{cutrank}(X_i, U_2^{(i)}) + \operatorname{cutrank}(U_1^{(i)}, U_2^{(i)}) \\
    &< \operatorname{cutrank}(X_i, Y_i) + r,
\end{align*}
that recursively
\begin{align*}\operatorname{cutrank}(X_8, Y_8) &= \operatorname{cutrank}(X_7\cup U_1^{(7)}, U_2^{(7)}) \\ &\leq \operatorname{cutrank}(X_7, U_2^{(7)})+\operatorname{cutrank}(U_1^{(7)}, U_2^{(7)}) \\ &< \operatorname{cutrank}(X_7, U_2^{(7)})+r \\ &< \operatorname{cutrank}(X_7, Y_7)+r  \\ &< \operatorname{cutrank}(X_0, Y_0)+8r,\end{align*}
as removing vertices from one side can never increase the rank. \\

Therefore, $\operatorname{cutrank}$($X_8, Y_8$) $< 60r + 8r = 68r$. Further, from an earlier argument $\tau$ chooses $X_8$ to be the small side of $(X_8, Y_8)$. Then, since $|X| = |A| + |B|$ and $|X_8| \geq |B| + \sum_{i=0}^7 |U_1^{(i)}|$, noting that $(U_1^{(i)}, U_2^{(i)})$ is balanced in $G[Y_i]$ gives that $|U_1^{(i)}| \geq |Y_i|/3$. This gives a recurrence in the form of \[|Y_{i+1}| = |U_2^{(i)}| \leq \frac{2}{3}|Y_i|.\] So, $|Y_i| \leq (2/3)^i |Y_0|$ and
\begin{align*}|U_1^{(i)}| &\geq \frac{1}{3}|Y_i| \\
&\geq \frac{1}{2}\left(\frac{2}{3}\right)^i |Y_0| \\
&= \frac{1}{3}\left(\frac{2}{3}\right)^i |A\cup Y|.\end{align*} Together, this gives us that 
\begin{align*}|X_8| &\geq |B| + \frac{1}{3}\sum_{k=0}^7\left(\frac{2}{3}\right)^k|A\cup Y| \\ &\geq |B| + 0.96|A\cup Y|\end{align*}
since we remove at least $1/3$rd of the vertices of $U_2^i$ between successive iterations for $i\in\{0,\dots,7\}$. Furthermore, we have that \[|X|=|B\cup A|=|B|+|A|.\]
To obtain a contradiction, we show that $|X_8|> |X|$ which is equivalent to $|Y|\geq 0.04|A|$ as $A\cap Y=\emptyset$. Specifically, for our contradiction to hold, we need to show \begin{align*}|X_8| &> |B| + 0.96 |A\cup Y| \\
&= |B| + 0.96(|A| + |Y|).\end{align*} Since $|X| = |A| + |B|$, this is equivalent to showing $0.96|Y| \geq 0.04|A|$ (or $|Y| > \frac{|A|}{24}$).\\

However, since we show (the stronger) $|A|\leq 6s = 18|Y|$, the contradiction holds, i.e., we have derived that $X_8$ is a valid side of a partition of $V(G)$ such that $|X_8|>|X_0|$, which is a contradiction. \\

Therefore, we conclude that $(X,Y)$ is $(3|Y|, 1/7)$-well-behaved. \qedhere\\
\end{proof}

\subsection{The Compressed Subgraph $H$}
We now compress the small side without increasing the relevant cutranks. Our strategy is to find a contradiction to the maximality of $X$. To do this, we must analyze the connections between $X$ and $Y$. We begin by compressing $X$ into a smaller, representative set $X^*$ without losing the essential rank information, as justified by the neighborhood equivalence classes discussed in Section \ref{sec: equivalence class neighborhood}.\\

\begin{claim}
\label{claim:findingH}
    There exists an induced subgraph $H\subseteq G$ which is obtained from $G$ by deleting some vertices from $X$ such that for  all $v\in X$, there exists at least one vertex $v'\in X\cap V(H) \eqqcolon X^*$ such that $N(v)\cap Y = N(v')\cap Y$.
\end{claim}
\begin{proof}
We choose $H$ such that for every vertex $v'\in X$, there is at least one vertex $v \in X\cap V(H)$ so that $N(v)\cap Y = N(v')\cap Y$, thus satisfying the condition and giving $\operatorname{cutrank}(H\cap X, Y) = \operatorname{cutrank}(X, Y)$. In other words, $H$ is constructed by selecting one vertex per equivalence class (as defined in Section \ref{sec: equivalence class neighborhood}).
\qedhere\\
\end{proof}

Having replaced $X$ by $X^*$ (one representative per class), we obtain an induced subgraph $H$ with $\operatorname{cutrank}(X^*, Y) = \operatorname{cutrank}(X,Y) < 70r$. We now form an induced $H'$ in which every balanced low‑rank cut must split $Y$. This forces any balanced separation to interact non-trivially with the uncompressed section, $Y$.\\

\begin{lemma}[Splitting $Y$]
    \label{lemma: splitting Y reduction}
    There exists an induced subgraph $H'\subseteq H$ and a balanced partition $(A', B')$ of $H'$ such that $A' \cap Y \neq \emptyset$ and $B'\cap Y\neq\emptyset$, $\operatorname{cutrank}_{H'}(A', B') < r$, and if $X^{**} \eqqcolon V(H') \cap X^*$, then \[\operatorname{cutrank}_{H'}(X^{**}, Y) \leq \operatorname{cutrank}_H(X^*, Y) = \operatorname{cutrank}_G(X,Y) < 70r.\]
\end{lemma}
\begin{proof}
    Pick any subset $X^{**}\subseteq X^*$ with $|X^{**}| < |Y|/2$ and set $H' \coloneqq H[Y\cup X^{**}]$. Then, $\frac{|Y|}{|V(H')|} = \frac{|Y|}{|Y|+|X^{**}|}$ is greater than $\frac{2}{3}$ if and only if $|Y|>2|X^{**}|$. Since no side of a balanced separation of $H'$ can contain all of $Y$ (because that side would exceed $2|V(H')|/3$), every balanced partition $(A', B')$ of $H'$ necessarily satisfies $A'\cap Y \neq \emptyset$ and $B' \cap Y \neq \emptyset$. \\
    
    By applying the contradiction hypothesis to the induced subgraph $H'$, there exists a balanced separation $(A', B')$ of $H'$ with $\operatorname{cutrank}_{H'}(A', B') < r$, which then splits $Y$ by the aforementioned argument. Next, since $X^{**}\subseteq X^*$, taking the submatrix induced by $X^{**}$ can only reduce the rank of the $X^*\times Y$ biadjacency matrix; therefore, we have
    \begin{align*}
        \operatorname{cutrank}_{H'}(X^{**}, Y) &\leq \operatorname{cutrank}_H(X^*, Y) \\
        &= \operatorname{cutrank}_G(X,Y),
    \end{align*}
    and Claim \ref{claim:rank-tangle existence} gives $\operatorname{cutrank}_G(X, Y) < 70r$.\qedhere\\
    \end{proof}

\noindent To control the rank of cuts within the subgraph $H'$, we first need a general way to relate the rank of a matrix to the rank of its sub-blocks. The following general lemmas provide such inequalities.\\

\begin{lemma}[Frobenius rank inequality \cite{doi:10.1137/1.9781611977448}] \label{lemma: matrix block inequality}
    For any conforming triple of matrices $(\mathbf{U}, \mathbf{V}, \mathbf{W})$, we have \[\rank(\mathbf{U}\mathbf{V}) + \rank(\mathbf{V}\mathbf{W}) \leq \rank(\mathbf{V}) + \rank(\mathbf{U}\mathbf{V}\mathbf{W}).\]
\end{lemma}

\begin{lemma}\label{lemma: matrix block rank inequality}
For a block matrix $\mathbf{M}$ over $\mathbb{F}_2$ with conforming matrices $(\mathbf{A}, \mathbf{B}, \mathbf{C}, \mathbf{D})$ given by \[\mathbf{M} \coloneqq
\begin{bmatrix}
\mathbf{A} & \mathbf{B}\\ \mathbf{C} & \mathbf{D}
\end{bmatrix},
\] we have \[\rank(\mathbf{M}) + \rank(\mathbf{B}) \geq \rank(\mathbf{A}) + \rank(\mathbf{D}).\]
\end{lemma}

\begin{proof}
Take
\[
\mathbf{U}=\begin{bmatrix}\mathbf{I} & \mathbf{0}\end{bmatrix}, 
\qquad
\mathbf{V}=\mathbf{M}=\begin{pmatrix}\mathbf{A}&\mathbf{B}\\ \mathbf{C}&\mathbf{D}\end{pmatrix},
\qquad
\mathbf{W}=\begin{bmatrix}\mathbf{0}\\ \mathbf{I}\end{bmatrix},
\]
where the identity blocks have the sizes of $\mathbf{A}$ and $\mathbf{D}$ respectively.
Then
\[
\mathbf{UV}=\begin{bmatrix}\mathbf{A} & \mathbf{B}\end{bmatrix}, 
\qquad 
\mathbf{VW}=\begin{bmatrix}\mathbf{B}\\ \mathbf{D}\end{bmatrix},
\qquad
\mathbf{UVW} = \mathbf{B}.
\]
Plugging these into Lemma \ref{lemma: matrix block inequality} with Frobenius' inequality gives
\begin{align*}\rank(\mathbf{M}) + \rank(\mathbf{B}) &= \rank(\mathbf{V}) + \rank(\mathbf{UVW}) \\ &\geq \rank(\mathbf{UV}) + \rank(\mathbf{VW}) \\
&= \rank\left(\begin{bmatrix}\mathbf{A} & \mathbf{B}\end{bmatrix}\right) + 
\rank\left(\begin{bmatrix}\mathbf{B} \\ \mathbf{D}\end{bmatrix}\right) \\
&\geq \rank(\mathbf{A}) + \rank(\mathbf{D}),
\end{align*}
where the last inequality follows by noting that the rank is non-decreasing when taking a submatrix (by removing columns/rows), and therefore 
\[\rank(\begin{bmatrix}\mathbf{A} & \mathbf{B}\end{bmatrix}) \geq \rank(\mathbf{A})\] and \[\rank\left(\begin{bmatrix}\mathbf{B} \\ \mathbf{D}\end{bmatrix}\right) \geq \rank(\mathbf{D}).\]
Together, these yield the desired inequality.\qedhere\\
\end{proof}

 With these algebraic tools, we now apply them to the specific cut $(A', B')$ within the subgraph $H'$. The following lemma uses the block-matrix inequality from Lemma \ref{lemma: matrix block rank inequality} to derive a crucial upper bound on the ranks of the sub-partitions within $H'$.\\

\begin{lemma} We have that \[\operatorname{cutrank}(X,Y)\geq \operatorname{cutrank}(A\cap X,A\cap Y) + \operatorname{cutrank}(B\cap X,  B\cap Y) - \operatorname{cutrank}(A, B).\]\label{helper lemma}
\end{lemma}
\begin{proof}

Let $X_1\coloneqq A\cap X, X_2\coloneqq B\cap X, Y_1\coloneqq A\cap Y$ and $Y_2\coloneqq B\cap Y$. \\

Next, construct the biadjacency matrix $\mathbf{T} \in \mathbb{F}_2^{(|X_1|+|X_2|)\times (|Y_1| + |Y_2|)}$ given by \[\mathbf{T} = \begin{bmatrix}
    \mathbf{A} & \mathbf{B} \\ \mathbf{C} &  \mathbf{D}
\end{bmatrix},\]
such that the rows of $\mathbf{T}$ are $X_1$ concatenated with $X_2$ and the columns of $\mathbf{T}$ are $Y_1$ concatenated with $Y_2$; here, $\mathbf{A}$ denotes the submatrix $\mathbf{T}[X_1, Y_1]$, $\mathbf{B}$ denotes the submatrix $\mathbf{T}[X_1, Y_2], \mathbf{C}$ denotes the submatrix $\mathbf{T}[X_2, Y_1]$, and $\mathbf{D}$ denotes the submatrix $\mathbf{T}[X_2, Y_2]$. \\

Over $\mathbb{F}_2$, Lemma \ref{lemma: matrix block rank inequality} gives us the rank formula \[\rank(\mathbf{T}) \geq \rank(\mathbf{A}) + \rank(\mathbf{D}) - \rank(\mathbf{B}).\] Translating this to cutrank, by using the fact that the cutrank is defined as the rank of the corresponding biadjacency matrix over $\mathbb{F}_2$, and noting that $\operatorname{cutrank}(A,B) \geq \rank(\mathbf{B})$ (adding to the rows or columns can never decrease the rank), we get \[\operatorname{cutrank}(X,Y)\geq \operatorname{cutrank}(X_1, Y_1) + \operatorname{cutrank}(X_2, 
Y_2) - \operatorname{cutrank}(A,B),\] which proves the lemma.\qedhere\\
\end{proof}

\begin{lemma}\label{lemma: second_condition}
If $(A',B')$ is a balanced partition of $H'$ (from Lemma \ref{lemma: splitting Y reduction}) such that $\operatorname{cutrank}_{H'}(A',B') < r$, then \[\operatorname{cutrank}_{H'}(A' \cap X^{**}, A' \cap Y) + \operatorname{cutrank}_{H'}(B'\cap X^{**}, B' \cap Y) \leq 71r - 2.\]
\end{lemma}

    \begin{proof}Let $X^{**}$ denote the set of representative vertices from $X$ that are present in the induced subgraph $H'$, as established in Lemma \ref{lemma: splitting Y reduction}. Specifically, $X^{**} = V(H') \cap X^*$, where $X^* = V(H) \cap X$. Note that by Lemma \ref{lemma: splitting Y reduction}, $\operatorname{cutrank}_{H'}(X^{**}, Y)\leq \operatorname{cutrank}_G(X, Y) < 70r$.  From Claim \ref{claim:rank-tangle existence}, we know that $\mathrm{cutrank}_G(X,Y) \leq 70r - 1$ since ranks are integral measures and $\operatorname{cutrank}_G(X,Y)<70r$. \\

Next, consider the biadjacency matrix $\mathbf{M}$ over $\mathbb{F}_2$ representing the cut between $X^{**}$ and $Y$ in the subgraph $H'$, block-partitioned by the balanced partition $(A',B')$ of $V(H')$,
\[
\mathbf{M} = \begin{bmatrix}
\mathbf{M}_{A'A'} & \mathbf{M}_{A'B'} \\
\mathbf{M}_{B'A'} & \mathbf{M}_{B'B'}
\end{bmatrix},
\]
where the rows of $\mathbf{M}$ are $(A'\cap X^{**}, B'\cap X^{**})$ and the columns are $(A'\cap Y, B'\cap Y)$. Specifically, $\mathbf{M}_{A'A'} = \mathrm{Adj}_{H'}(A' \cap X^{**}, A' \cap Y)$, $\mathbf{M}_{A'B'} = \mathrm{Adj}_{H'}(A' \cap X^{**}, B' \cap Y)$,  $\mathbf{M}_{B'A'} = \mathrm{Adj}_{H'}(B' \cap X^{**}, A' \cap Y)$, and $\mathbf{M}_{B'B'} = \mathrm{Adj}_{H'}(B' \cap X^{**}, B' \cap Y)$. The ranks of these matrices correspond to the cutrank of the respective vertex partitions. \\

We are given that $\mathrm{cutrank}_{H'}(A',B')\leq r-1$ (due to the integrality of the rank), and the matrices representing this cut are $\mathrm{Adj}_{H'}(A', B'\cap V(H'))$ and $\mathrm{Adj}_H(B', A'\cap V(H'))$. \\

Moreover, by definition
\[\rank(\mathbf{M}) = \operatorname{cutrank}_{H'}(X,Y) \leq 70r - 1\] and \[\rank(\mathbf{M}_{A'B'}) \leq \operatorname{cutrank}_{H'}(A',B') \leq r - 1.\]
Applying Lemma \ref{lemma: matrix block rank inequality} to $\mathbf{M}$ yields
\[\rank(\mathbf{M}) + \rank(\mathbf{M}_{A'B'}) \geq \rank(\mathbf{M}_{A'A'}) + \rank(\mathbf{M}_{B'B'}).\]
Combining these gives that
\begin{align*}\operatorname{cutrank}_{H'}(A'\cap X^{**}, A'\cap Y) &+ \operatorname{cutrank}_{H'}(B'\cap X^{**}, B'\cap Y) \\
&= \rank(\mathbf{M}_{A'A'}) + \rank(\mathbf{M}_{B'B'}) \\
&\leq 70r - 1 + r - 1 \\
&= 71r - 2
\end{align*}
which proves the claim of the lemma. 
\qedhere\\
\end{proof}

 Bore lifting, we isolate an important size relation between $X$ and $Y$. This prevents a degenerate case where the small side has a cardinality that is too small to prevent a contradiction from the tangle maximality.

\begin{lemma}
    $|X|\geq 18|Y|$.\label{lemma: x>18y}
\end{lemma}
\begin{proof}
    We apply the definition of $(s, \epsilon)$-well-behavedness (in Definition \ref{def: well-behaved separations}) to the trivial partition $(A, \emptyset)$ of $X$, with $s=3|Y|$ and $\epsilon=1/7$. For this, note that $|A|=|X|$ and
\[\operatorname{cutrank}(A,Y) = \operatorname{cutrank}(X,Y) \leq 70r\]
If $|X| \leq 6s = 18|Y|$, then the well-behaved property would require $\operatorname{cutrank}(A \cup Y, B) = \operatorname{cutrank}(X \cup Y, \emptyset) = 0 \geq 6\epsilon \cdot 70r = 60r$,
which is clearly false since $r \geq 1$. \\

Therefore, we must have $|X| > 6s = 18|Y|$, which proves the lemma.
    \qedhere\\
\end{proof}

\subsection{Contradiction via a Low-Rank Cut}
We are now ready to lift the balanced low‑rank cut $(A', B')$ of $H'$ to two separations of $G$ by adjoining each piece of $Y$ back to $X$. Subadditivity bounds ensure that each lifted separation has cutrank strictly less than $72r$.\\

\begin{lemma}[Lifting balanced cuts] \label{lemma:cutrank<r implies cutrank<70r}
    For the induced subgraph $H' \subseteq G$ from Claim \ref{claim:findingH}, and its balanced partition $(A',B')$ (from Lemma \ref{lemma: splitting Y reduction}) where $\operatorname{cutrank}_{H'}(A',B')<r$, we have that 
    \[\operatorname{cutrank}_G(X\cup (A' \cap  Y), B' \cap  Y) < 72r\] and \[\operatorname{cutrank}_G(X\cup (B' \cap Y), A' \cap Y) < 72r.\]
\end{lemma}

\begin{proof}
Let $(A',B')$ be the balanced partition of $H'$ from Lemma \ref{lemma: splitting Y reduction}. Let $X^{**}=X\cap V(H')$ and $\mathbf{M}=\mathrm{Adj}_{H'}(X^{**},Y)$, block-partitioned by $(A',B')$ as
\[\mathbf{M}=\begin{bmatrix}\mathbf{M}_{A'A'} & \mathbf{M}_{A'B'} \\ \mathbf{M}_{B'A'} & \mathbf{M}_{B'B'}\end{bmatrix}.
\]
By Lemma \ref{lemma: second_condition} and Claim \ref{claim:findingH}, we have that
\[\operatorname{rank}(\mathbf{M})=\operatorname{cutrank}_{H'}(X^{**},Y)\le 70r-1.
\]
Moreover, we have the symmetric relations 
\[\operatorname{rank}(\mathbf{M}_{A'B'})\le \operatorname{cutrank}_{H'}(A',B')\le r-1\] 
and
\[\operatorname{rank}(\mathbf{M}_{B'A'})\le \operatorname{cutrank}_{H'}(A',B')\le r-1.\] 

\noindent Since $B'\subseteq X^{**} \cup Y$ and $X^{**} \subseteq X$, we have $X \cup B' = X \cup (B' \cap Y)$. Similarly, $X \cup A' = X \cup (A' \cap Y)$; therefore,  we have that the two lifted partitions of $V(G)=X\cup Y$ are \[P_1\coloneqq (X\cup (A'\cap Y), B'\cap Y)\] and \[P_2\coloneqq (X\cup (B'\cap Y), A'\cap Y).\] 

We first bound the cutrank of $P_1$. By subadditivity of the matrix rank, we write
\begin{align*}
    \operatorname{cutrank}_G(X\cup (A'\cap Y), B'\cap Y) &\leq \operatorname{cutrank}_G(X, B'\cap Y) + \operatorname{cutrank}_G(A'\cap Y, B'\cap Y) \\
    &\leq \operatorname{cutrank}_G(X, Y) + \operatorname{cutrank}_G(A'\cap Y, B'\cap Y) \\
    &\leq \operatorname{cutrank}_G(X,Y) + \operatorname{cutrank}_{H'}(A',B') \\
    &\leq (70r - 1) + (r-1) \\
    &< 72r,
\end{align*}

where the second inequality follows from noting that $\mathrm{Adj}_G(X, B'\cap Y)$ is a column submatrix of $\mathrm{Adj}_G(X,Y)$, and therefore $\operatorname{cutrank}_G(X,B'\cap Y) \leq \operatorname{cutrank}_G(X, Y)$. The third inequality follows from noting that $\mathrm{Adj}_G(A'\cap Y, B'\cap Y)$ is a submatrix of $\mathrm{Adj}_{H'}(A', B')$, whose cutrank is bounded (by supposition) by $r$, and the fourth inequality follows by Claim \ref{claim:rank-tangle existence} which gives $\operatorname{cutrank}_G(X, Y) < 70r$.\\

We now bound the cutrank of $P_2 \coloneqq (X\cup (B\cap Y), A\cap Y)$. For this, a symmetric argument gives that 
\begin{align*}
    \operatorname{cutrank}_G(X\cup (B'\cap Y), A'\cap Y) &\leq \operatorname{cutrank}_G(X, A'\cap Y) + \operatorname{cutrank}_G(B'\cap Y, A'\cap Y) \\
    &\leq 70r - 1 + (r-1) \\
    &< 72r, 
\end{align*}
which proves the claim.\qedhere\\
\end{proof}

Now we are ready to complete the proof of the theorem. By way of contradiction, suppose that every induced subgraph $H' \subseteq G$ has $\operatorname{min-bal-cutrank}(H')$ strictly less than $r$, i.e., for every balanced partition $(A', B')$ of $V(H')$ we have $\operatorname{cutrank}_{H'}(A', B') < r$. Let $\tau$ be a rank-tangle of order $72r$. By Claim \ref{claim:rank-tangle existence}, fix $(X,Y)$ with $\mathrm{cutrank}(X,Y)<70r$ where $X\in\tau$, that is $(3|Y|,1/7)$-well-behaved. By Lemma \ref{lemma: splitting Y reduction}, there exists an induced subgraph $H'\subseteq G$ and a balanced partition $(A', B')$ of $H'$ such that $\mathrm{cutrank}_{H'}(A',B')<r$, $A'\cap Y \neq \emptyset$ and $B'\cap Y\neq \emptyset$, and if $X^{**} = V(H')\cap X^*$, then $\operatorname{cutrank}_{H'}(X^{**}, Y) < 70r$. \\

Now, consider the two lifted partitions of $V(G)$ based on this cut: $P_1 \coloneqq (X\cup (A'\cap Y), B'\cap Y)$ and $P_2\coloneqq (X\cup (B'\cap Y), A'\cap Y)$. 
By Lemma \ref{lemma:cutrank<r implies cutrank<70r}, which applies because $\operatorname{cutrank}_{H'}(A', B') < r$ and $\operatorname{cutrank}_{H'}(X^{**}, Y) < 70r$ (which implies $\operatorname{cutrank}_G(X,Y) < 70r$), both of these lifted separations $P_1$ and $P_2$ have cutrank strictly less than $72r$. \\

Therefore, the tangle $\tau$ must orient both: by the second tangle axiom in Definition \ref{def: rank-tangle of order k}, the three sets $X, B'\cap Y$, and $A'\cap Y$ cannot all be in $\tau$ (as their union is $V(G)$). Since $X\in \tau$, at least one of $(B'\cap Y)$ or $(A'\cap Y)$ is not in $\tau$. Now, if $B'\cap Y \notin \tau$, then its complement in $P_1$, which is $X\cup (A'\cap Y)$, must be in $\tau$. Similarly, if $A'\cap Y\notin \tau$, then $X\cup (B'\cap Y)$ must be in $\tau$. In either case, $\tau$ must choose a set that is strictly larger than $X$ (as Lemma \ref{lemma: splitting Y reduction} guarantees that $A'\cap Y \neq \emptyset$ and $B'\cap Y \neq \emptyset$); however, this then contradicts the maximality of $X$ from Claim \ref{claim:rank-tangle existence}. Therefore, our initial assumption must be false, and there exists an induced subgraph of $G$ with min-bal-cutrank at least $r$, which proves Theorem \ref{thm:main}.\qedhere
\end{proof}

\section{Conclusion}

We show that any graph with sufficiently large rankwidth must contain an induced subgraph where every balanced bisection has a high cutrank. In its dual, this says that a graph has small rankwidth if every large induced subgraph can be bisected by a low-rank cut, and that it has large rankwidth if it contains a large induced subgraph that is robustly connected against all balanced low-rank cuts. This establishes an equivalence between rankwidth and the maximum minimum-balanced-cutrank over all induced subgraphs. 

\section{Future Work} We identify several key areas for future work in this field. \\

First, it would be interesting to identify an extension of dense-graph expansion to \emph{edge-weighted graphs}. For this, we envision interesting roadblocks in the bit complexity of the adjacency graphs \cite{anand2024bitcomplexitydynamicalgebraic} and their geometry \cite{gotlib2023highperspectivehighdimensional} to be overcome. 
Secondly, one reason why expansion is a mathematically interesting notion is because its linear-algebraic, topological, and combinatorial notions are all equivalent \cite{Hoory2006}. It would be interesting to identify similar equivalencies with rank-expansion, and to derive connections to information-theory; for instance, our framework is built on the rank of matrices over $\mathbb{F}_2$, but a generalization would be to see if analogous results hold for matrix ranks over $\mathbb{R}$ or other finite fields. 
Thirdly, the constant $70$ in our main result in Theorem \ref{thm:main} arises from the constants in the ``well-behaved'' definition and the number of iterations in the argument from \cite{Dvorak_Norin_2018}. We believe this is not tight, and determining the optimal constant relating rankwidth and the maximum min-bal-cutrank is a natural follow-up problem. 
Finally, certain algorithms admit a structural parameterization through the VC-dimension \cite{pmlr-v49-alon16, anand2025structuralcomplexitymatrixvectormultiplication,chan2025trulysubquadratictimealgorithms,duraj2024betterdiameteralgorithmsbounded, karczmarz2024subquadraticalgorithmsminorfreedigraphs} and the related Pollard pseudodimension. Similarly, it would be interesting if our structural characterization could lead to new approximation algorithms or FPT algorithms for problems parameterized by rankwidth.\\

\section{Acknowledgements}
This work was supported by NSF Grants CCF 2338816 and DMS 2202961. EA thanks Professor Rose McCarty for introducing me to this problem and the work in this field, for her deep insights and early discussions, and for hosting him at Princeton University's mathematics department to work on this problem. EA also expresses thanks to Dr. James Davies and Professor Jan van den Brand for insightful early discussions this work, Elia Gorokhovsky for suggestions on the proofs of certain lemmas in the manuscript, and Ethan Leventhal for helpful comments on the presentation of the manuscript.\\

\bibliographystyle{alpha}
\bibliography{main.bib}

\end{document}